\documentclass{amsart}
	\usepackage{aliascnt}
	\usepackage{hyperref}
	\usepackage{amssymb}	
	\usepackage{amsmath}
	\usepackage{MnSymbol}

	\newaliascnt{lemma}{thm}
	\newtheorem{lemma}[lemma]{Lemma}  
	\aliascntresetthe{lemma}

	\newaliascnt{prop}{thm}
	\newtheorem{prop}[prop]{Proposition} 
	\aliascntresetthe{prop}

	\newaliascnt{defn}{thm}
	\newtheorem{defn}[defn]{Definition}
	\aliascntresetthe{defn}

	\newaliascnt{cor}{thm}
	
	\aliascntresetthe{cor}

	\newaliascnt{rem}{thm}
	
	\aliascntresetthe{rem}

	\newaliascnt{exm}{thm}
	
	\aliascntresetthe{exm}

	\usepackage{layout}
	\newcommand{\rk}{\operatorname{rk}}
	\newcommand{\Sy}{\operatorname{S}}
	\newcommand{\Ker}{\operatorname{Ker}}
	\newcommand{\Ps}{\mathbb{P}}
	\newcommand{\ins}{\minushookup}
	
	\newcommand{\Span}[1]{\left\langle\,#1\,\right\rangle}
	
\begin{document}

\title[A Proof that the Maximal Rank for Plane Quartics is Seven]{A Proof that the Maximal Rank\\for Plane Quartics is Seven}
\author{Alessandro De Paris}

\begin{abstract}
At the time of writing, the general problem of finding the maximal Waring rank for homogeneous polynomials of fixed degree and number of variables (or, equivalently, the maximal symmetric rank for symmetric tensors of fixed order and in fixed dimension) is still unsolved. To our knowledge, the answer for ternary quartics is not widely known and can only be found among the results of a master's thesis by Johannes Kleppe at the University of Oslo (1999). In the present work we give a (direct) proof that the maximal rank for plane quartics is seven, following the elementary geometric idea of splitting power sum decompositions along three suitable lines.
\end{abstract}

\maketitle



\section{Introduction}

The (Waring) \emph{rank} of a homogeneous polynomial is the minimum number of summands needed to express it as a sum of powers of linear forms.  According to \cite{G}, the problem of finding the maximal rank for polynomials of fixed degree $d$ and number $n$ of variables may be called \emph{little Waring problem} for polynomials, in analogy with the classical problem in number theory. The \emph{big Waring problem} is a `generic version', with a solution that was given by a now classical theorem of Alexander and Hirschowitz (see, e.g., \cite[Theorem~3.2.2.4]{L}). To our knowledge, contrary to the number theoretic situation, the little Waring problem for polynomials is solved only for a few values of $n,d$.

Beyond the interest due to its connection with the classical Waring problem, this topic deserves attention as a part of tensor theory. Questions about tensors are attracting researchers because of the recent discovery of new applications (see \cite{L}). In that respect, the focus is mainly on low-rank and general (not necessarily symmetric) tensors. Nevertheless, high-rank symmetric tensors may well provide some useful insights in a field that, in spite of its long history and great recent efforts devoted to it, seems far from being completed.

In the relatively recent book \cite{L} (see the preamble to Chap.~9), possible ranks and border ranks are reported to be known only when $n=2$ or $d=2$ or $n=d=3$. A thorough study of the case $n=3$, $d=4$ is one of the main subjects of \cite{BGI}. Note that Theorem~44 of that paper would have probably been easily completed with the description of $\sigma_6\left(X_{2,4}\right)\backslash\sigma_5\left(X_{2,4}\right)$ if the fact that the maximal rank for plane quartic is seven had been known. For these reasons, the author wondered about that maximal rank. During the investigation, Kleppe's thesis \cite{K} was brought to our attention by Edoardo Ballico. We admit not having thoroughly checked that thesis, but we have good reasons to say that its results are highly reliable. In particular, the maximal rank for plane quartics is seven. However, we also have good reasons to believe that the approach we follow in the present paper significantly differs and is worthy of consideration. Some of Kleppe's results are also involved in the proof of a general bound for the rank of polynomials that is presented in~\cite{J} (see also \cite{BB}, \cite{BB2}). Note that the formula in that work gives a bound of nine for plane quartics. Hence, in view of the search for better general bounds and therefore in view of the little Waring problem for polynomials, a deeper understanding of the case of plane quartics may be useful.

Our basic idea is to look for summands that are forms in a lesser number of variables. In order to provide more details, let us first fix some standing conventions. An algebraically closed field $K$ of zero characteristic is fixed throughout the paper. The symmetric algebra of a $K$-vector space $V$ will be denoted by $\Sy^\bullet V$, with degree $d$-components denoted by $\Sy^dV$ and with the convention that they vanish for $d<0$. We also assume $\Sy^1V=V$. We keep fixed the notation $S^\bullet$, $S_\bullet$ for two such symmetric algebras on which a perfect pairing $a_1:S^1\times S_1\to K$ of $K$-vector spaces is tacitly assigned (of course, $S^d$ and $S_d$ are the degree $d$ homogeneous components of $S^\bullet$, $S_\bullet$, respectively). The perfect pairing induces the \emph{apolarity (perfect) pairing}
\[
a_d:S^d\times\ S_d\to K
\]
in each fixed degree $d$, in a natural way. Namely, it is uniquely determined by the condition
\[
a_d\left(\;x^1\cdots x^d\,,x_1\cdots x_d\;\right):=\operatorname{perm}\left(\begin{array}{ccc}a_1\left(x^1,x_1\right)&\cdots&a_1\left(x^1,x_d\right)\\\\\vdots&\ddots&\vdots\\\\a_1\left(x^d,x_1\right)&\cdots&a_1\left(x^d,x_d\right)\end{array}\right)\;,
\] 
for all $x^1\ldots x^d\in S^1$, $x_1\ldots x_d\in S_1$, where $\operatorname{perm}$ denotes the \emph{permanent} (a `signless determinant': $\operatorname{perm}\left(x^i_j\right):=\sum_{\sigma}x^i_{\sigma(i)}$, with $\sigma$ ranging over all permutations of the indices).

Given $s\in S^\delta$, $x\in S_d$, there exists a unique $y\in S_{d-\delta}$ such that
\[
a_{d-\delta}(t,y)=a_d(st,x)\;,\qquad\forall t\in S^{d-\delta}\;,
\]
because $a_{d-\delta}$ is a perfect pairing. We call the element $y$ the \emph{contraction} of~$x$ by~$s$ (it vanishes when $\delta>d$), and the definition extends by additivity for all $s\in S^\bullet,x\in S_\bullet$. We allow ourselves to borrow from the context of exterior algebras the notation for contraction:
\[
y=:s\ins x\;.
\]
It is convenient to keep in mind two (well-known) basic rules for contractions:
\[
st\ins x=s\ins\left(t\ins x\right)
\]
and
\begin{equation}\label{Li}
s\ins xy=(s\ins x)y+x(s\ins y)\;.
\end{equation}
From these rules we recover a very common description of the rings $S^\bullet$, $S_\bullet$: they are usual polynomial rings, $S^\bullet$ is usually denoted by $T=k\left[y_1,\ldots ,y_n\right]$, and its elements act as (constant coefficients) differential operators on the polynomials of $S_\bullet=:S=k\left[x_1,\ldots ,x_n\right]$. Our alternative notation aims at suggesting a geometric viewpoint from which (homogeneous) elements of $S_\bullet$ are `contravariant' and give `multipoints' in a projective space $\Ps$, meanwhile that of $S^\bullet$ are `covariant' and give hypersurfaces in $\Ps$ (nevertheless, at a technical level, the roles of the two rings are perfectly symmetric). Still in view of our more elementary geometric viewpoint, we shall use the orthogonality sign $\perp$ with reference to the original pairing $S^1\times S_1\to K$ only (and not for the apolar ideal). Therefore, $\left\{x,y\right\}^\perp$, with $x,y\in S_1$, will denote the set of $l\in S^1$ that vanish at $x,y$ (when viewed as linear forms, that is, $l\ins x=l\ins y=0$). For instance, in \cite{K} our $\left\{x,y\right\}^\perp$ would be denoted by $\left(x,y\right)^\perp_1$ (the first homogeneous component of the apolar ideal of $(x,y)\subset S$).

We shall not use angle parentheses to denote the apolarity pairings, because we are more comfortable with using them to indicate the linear span of a set of vectors. We prefer to (formally) look at points in the projective space $\Ps\left(S_1\right)$ as one-dimensional subspaces $\Span{x}\subseteq S_1$, $x\ne 0$. When a scheme structure is needed, $\Ps\left(S_1\right)$ may be (naturally) replaced by $\operatorname{Proj}S^{\bullet}$ (and $\Span{x}$ by the ideal generated by $\{x\}^\perp$).

Finally, we shall sometimes make use of the \emph{partial polarization map} $f_{\delta,d}:S^\delta\to S_d$ of $f\in S_{d+\delta}$ (cf., e.g., \cite[2.6.6]{L}). It is simply defined by
\[
f_{d,\delta}(t):=t\ins f
\]
and we shall keep the notation $f_{d,\delta}$.

Now that our standing notation is set up, let us quickly describe the idea to bound the rank we are following. In the case of a ternary quartic $f\in S_4$ ($\dim S_1=3$) this is easy. Indeed, let us consider in $\Ps S^3$ the closed subvariety $X$ consisting of cubics that are broken in three lines:
\[
X:=\left\{\Span{x^0x^1x^2}:\;x^0,x^1,x^2\in S^1-\{0\}\right\}\;.
\]
Consider also the subspace
\[
Y:=\left\{\Span{c}:\;c\in S^3-\{0\}, c\ins f=0\right\}\;.
\]
We have $\dim\Ps S^3=9$, $\dim X=6$, $\dim Y\ge 6$, so that $\dim\left(X\cap Y\right)\ge 3$. Hence we can always find $x^0,x^1,x^2\in S^1-\{0\}$ such that $x^0x^1x^2\ins f=0$. We expect that, generically, $x^0,x^1,x^2$ should be linearly independent, and in this case we can write
\[
f=f_0\left(x_1,x_2\right)+f_1\left(x_0,x_2\right)+f_2\left(x_0,x_1\right)\;,
\]
with $x_0,x_1,x_2$ being the basis of $S_1$ dual to $\left(x^0,x^1,x^2\right)$.
Moreover, we have three degrees of freedom in the decomposition, due to the possibility of moving ${x_0}^4,{x_1}^4,{x_2}^4$ among $f_0,f_1,f_2$ (generically, one might exploit this to reach $\rk f_i\le 2$). We shall explain how to manage special cases later. The implementation of this idea will lead us to amusing exercises about very elementary objects, such as quadric cones in a three-dimensional space and the rational normal quartic curve. By solving them, in our opinion we gain a perspective from which the intricacies of high rank symmetric tensors can be usefully organized.

\section{Preparation}

Since we are building our main proof on the basis of quite elementary facts, even the moderately experienced reader will likely prefer to prove these facts as an exercise, instead of being bored by reading detailed proofs. That is why in this preliminary section we shall limit ourselves to the statements and a few hints.

First of all, let us recall that the rank stratification for binary forms, i.e., when $\dim S_1=2$, is well known (form a geometric viewpoint, it is based on properties of rational normal curves). Recent references are, among many others, \cite[9.2.2]{L}, \cite{CS}, as well as \cite[chap.~1]{K}. To begin with, we recall that for a binary quartic $f\in S_4$, $\rk f\le 4$. Moreover, the secant variety $X$ to the rational normal quartic curve $Q$ that consists of all $\Span{x^4}$ with $x\in S_1$, is a hypersurface in $\Ps S_4$. Its complement is exactly the set of all $\Span{f}$ with $\rk f=3$. The equation for $X$ is given by the condition $\det f_{2,2}=0$, and therefore $\deg X=3$. Points $\Span{f}$ of the \emph{tangent} variety, but that lies outside $Q$, are exactly those for which $\rk f=4$; the tangent to $\Span{x^4}\in Q$ is $\Ps\Span{x^4,x^3y}$ with $y\in S^1-\Span{x}$. We need now to describe the rank stratification of particular planes in $\Ps S_4$.

\begin{lemma}\label{Casi}
Let $\dim S_1=2$, $S_1=\Span{x_0,x_1}$, and $W$ be a subspace of $S_4$ with $\dim W=3$ and containing $L:=\Span{{x_0}^4,{x_1}^4}$. Set $A:=\Ps W-\Ps L$, which can be regarded as an affine plane with line at infinity $\Ps L$, and
\[
R:=\left\{\Span f\in A: \rk f\ne 3\right\}\;,\qquad
R':=\left\{\Span f\in A: \rk f=4\right\}\;.
\]
Then we have one of the following alternatives \ref{c11}, \ref{c12}, \ref{c2}:
\begin{enumerate}
\item\label{c1} $R'$ consists of at most two points and
	\begin{enumerate}
	\item\label{c11} $R\ne\emptyset$ is an affine conic with points at infinity exactly $\Span{{x_0}^4}$, $\Span{{x_1}^4}$, and when $R$ possesses a singular point $\Span{x}$ we have $\rk x=1$ and $R'=\emptyset$; or
	\item\label{c12} $R\ne\emptyset$ is an affine line with point at infinity different from $\Span{{x_0}^4}$, $\Span{{x_1}^4}$; or
	\end{enumerate}
\item\label{c2} $R=R'\ne\emptyset$ is an affine line with point at infinity either $\Span{{x_0}^4}$ or $\Span{{x_1}^4}$, and, more precisely, $R=R'=A\cap\Ps{\Span{{x_0}^4,{x_0}^3x_1}}$ in the first case, $R=R'=A\cap\Ps{\Span{x_0{x_1}^3,{x_1}^4}}$ in the other.
\end{enumerate}
\end{lemma}

The proof can safely be left to the reader, but we suggest to first keep in mind that points $\Span{f}\in\Ps W$ with $\rk f\ne 3$ constitute a reducible cubic curve with $\Ps L$ as a component. The following geometric considerations might also be helpful. Let us look at the projection \[\Ps S_4-\Ps L\;\longrightarrow\;\Ps\left(S_4/L\right)\;,\qquad\Span{x}\longmapsto\Span{\{x+L\}}\;,\] so that $\Ps W$ projects onto a point $P\in\Ps\left(S_4/L\right)$. Lines $\ell$ through $P$ in $\Ps\left(S_4/L\right)$ comes from hyperplanes of $\Ps{S_4}$ containing $\Ps W$. Such a hyperplane meets the rational normal quartic $Q$ (consisting of all $\Span{x^4}$, $x\in S_1$) in $\Span{{x_0}^4}$, $\Span{{x_1}^4}$, and further points $\Span{{x_2}^4}$, $\Span{{x_3}^4}$. If $\ell\ni P$ is not tangent (somewhere) to the projection $C$ of $Q$ (which is a conic), we have a secant line $\ell':=\Ps\Span{{x_3}^4,{x_4}^4}$ that must intersect the plane $\Ps W$, and of course if $\Span{x}\in\ell'\cap\Ps W$ then $\rk x\le 2$. Tangents give rank four (or rank one) forms instead. Conversely, any $\Span{x}\in A$ that lies on a secant or tangent line $\ell'$ to $Q$ that does not contain $\Span{{x_0}^4}$ or $\Span{{x_1}^4}$, must come in the previous way from the hyperplane joining $\ell'$ and $\Ps W$. With the above said in mind, let $P_0,P_1\in C$ be the projections  of (the tangents to $Q$ at) $\Span{{x_0}^4},\Span{{x_1}^4}\in Q$, and $\ell_0$ be the line through them. Then \hyperref[c2]{Case~\ref{c2}} occurs when $P$ comes to coincide with $P_0$ or $P_1$, and \hyperref[c12]{Case~\ref{c12}} occurs when $P\in\ell_0\,-\{P_0,P_1\}$. \hyperref[c11]{Case~\ref{c11}} occurs when $P\not\in\ell_0$, with $R$ being singular exactly when $P=\Span{x}$ also lies on $C$.

For ease of exposition, in this paper we make use of the following \textit{ad hoc} terminology, related to the situation of the above lemma.

\begin{defn}
Let $W$ be a $K$-vector space, $\dim W=3$, $y,z\in W$ distinct vectors and $R,R'\subset\Ps W-\Ps\Span{y,z}$. Throughout this paper we say that $\left(W,y,z,R,R'\right)$ is an \emph{R-configuration of type \ref*{c11}, \ref*{c12} or \ref*{c2}}, if it fulfils the corresponding condition in \autoref{Casi} with $y,z$ in place of ${x_0}^4,{x_1}^4$ (without reference to $x$ in \hyperref[c11]{Case~\ref{c11}} nor to the polynomial description of $R=R'$ in \hyperref[c2]{Case~\ref{c2}}).
\end{defn}

\begin{prop}\label{3D}
Let $\mathcal{C}_i=\left(W_i,y_i,z_i,R_i,R'_i\right)$, $i\in\{0,1,2\}$, be R-configurations. Let $W$ be a $K$-vector space, $\dim W=4$, $w_0,w_1,w_2\in W$ linearly independent vectors, and $\alpha_i:W\twoheadrightarrow W_i$, $i\in\{0,1,2\}$, surjective linear maps such that for each $i$, $\alpha_i$ sends $w_i$ into $0$ and the other two vectors into $y_i,z_i$ (in whatever order, but one-to-one). Finally, for each $i$, let us consider the (affine) map
\[
\widehat{\alpha_i}:\Ps W-\Ps{\Span{w_0,w_1,w_2}}\longrightarrow\Ps{W_i}-\Ps{\Span{x_i,y_i}}\;,\qquad \Span{w}\mapsto\Span{\alpha_i(w)}
\]
and set
\[
\widehat{R_i}:=\widehat{\alpha_i}^{-1}\left(R_i\right)\;,\qquad\widehat{R_i'}:=\widehat{\alpha_i}^{-1}\left(R_i'\right)\;.
\]
If $\mathcal{C}_0$ and $\mathcal{C}_1$ are not of type~\ref*{c2} and
\begin{equation}\label{Vuoto}
\left(\widehat{R_0}\cap\widehat{R_1}\right)-\left(\widehat{R_0'}\cup\widehat{R_1'}\cup\widehat{R_2'}\right)=\emptyset\;,
\end{equation}
then the R-configuration $\mathcal{C}_2$ is of type~\ref*{c2}, one of the others, say $\mathcal{C}_j$, is of type~\ref*{c11} with $R_j$ a reducible conic, and $\widehat{R_2'}$ is a component (plane) of $\widehat{R_j'}$.
\end{prop}

Let us outline a way to organize a proof that to some extent avoids a cumbersome analysis. The dimension of each irreducible component of the intersection $X:=\widehat{R_0}\cap\widehat{R_1}$ is at least one. Considering $\Ps{\Span{w_0,w_1,w_2}}$ as the plane at infinity, it is easy to see that there must exist a component $Y$ of $X$ with a point $P$ at infinity that does not lie in the line $\Ps{\Span{w_0,w_1}}$. Note that $\widehat{R_0'}$ is a (possibly empty) union of lines through $\Span{w_0}$, and $\widehat{R_1'}$ a union of lines through $\Span{w_1}$. If $\mathcal{C}_2$ is not of type~\ref*{c2}, then the condition~\eqref{Vuoto} above would imply that $P=\Span{w_2}$ and that $Y$ is a line. But this is possible only when $\mathcal{C}_0$, $\mathcal{C}_1$ are of type~\ref*{c1}, with $R_0,R_1$ reducible conics. But recall that if $R_i$ is reducible, then $R'_i$ is empty, and note that when $R_0,R_1$ are reducible, the intersection $X$ must also contain two lines with points at infinity $\Span{w_0},\Span{w_1}$. Since that picture is incompatible with condition~\eqref{Vuoto}, we have that $\mathcal{C}_2$ must be of type~\ref*{c2} and that $\widehat{R_2'}$ must be a plane containing $Y$.

Now, the line at infinity of $\widehat{R_2'}$ is either $\Ps{\Span{w_0,w_2}}$ or $\Ps{\Span{w_1,w_2}}$, and let it be $\Ps{\Span{w_j,w_2}}$ with the appropriate $j\in\{0,1\}$. If $\widehat{R_2'}\cap\widehat{R_j}\ne\widehat{R_2'}$, then this intersection is a union of lines through $\Span{w_j}$; hence it can not contain $Y$ and \eqref{Vuoto} would fail. Therefore, $\widehat{R_2'}\subseteq\widehat{R_j}$ and henceforth $\widehat{R_j}$ is reducible. This immediately implies that also $R_j$ is reducible and that $\mathcal{C}_j$ is of type~\ref*{c11}.

Let us now state what happens when the situation of \autoref{Casi} degenerates `by collision' of $\Span{{x_0}^4}$ and $\Span{{x_1}^4}$.

\begin{lemma}\label{Casi2}
Let $\dim S_1=2$, $S_1=\Span{x_0,x_1}$, and $W$ be a subspace of $S_4$ with $\dim W=3$ and containing $L:=\Span{{x_0}^4,{x_0}^3x_1}$. Set $A:=\Ps W-\Ps L$, which can be regarded as an affine plane with line at infinity $\Ps L$, and
\[
R:=\left\{\Span f\in A: \rk f\ne 3\right\}\;,\qquad
R':=\left\{\Span f\in A: \rk f=4\right\}\;.
\]
Then $R'$ consists of at most two points, and we have one of the following alternatives \ref{c11}, \ref{c12}, \ref{c2}:
\begin{enumerate}
\item\label{d1} 
	\begin{enumerate}
	\item\label{d11} $R\ne\emptyset$ is an affine conic with one point at infinity $\Span{{x_0}^4}$ (hence, a `parabola'), and when this conic is degenerate we have that it is a (double) affine line, that there exist $\Span{x}\in R$ with $\rk x=1$ and that $R'=\emptyset$; or
	\item\label{d12} $R\ne\emptyset$ is a (simple) affine line with point at infinity different from $\Span{{x_0}^4}$; or
	\end{enumerate}
\item\label{d2} $R=R'=\emptyset$ and, moreover, $W=\Span{{x_0}^4,{x_0}^3x_1,{x_0}^2{x_1}^2}$.
\end{enumerate}
\end{lemma}

Finally, as a warm up, we present our approach to the problem in an easy situation (that will sometimes arise during the main proofs). At this early stage, the overlap of our arguments with those of \cite{K} is larger: see \cite[Theorem~3.6]{K} (in the case when $\{D_0=0\}$ is a union of distinct lines). Kleppe's proof is longer, but the statement of that theorem gives much more than an upper bound. Note also that under the hypothesis we are using below, that is, $x^0x^1\ins f=0$, Kleppe's theorem gives the best bound, that is, six.

\begin{prop}\label{Wu}
Let $\dim S_1=3$, $f\in S_4$. If there exist linearly independent $x^0,x^1\in S^1$ such that $x^0x^1\ins f=0$ then $\rk f\le 7$.
\end{prop}
\begin{proof}
Let us choose $x_2\in\{x^0,x^1\}^\perp-\{0\}$, $x_1\in\{x^0\,\}^\perp-\Span{x_2}$, $x_0\in\{x^1\,\}^\perp-\Span{x_2}$, and set
\[
V_0:=\Sy^4\Span{x_1,x_2}\;,\qquad V_1:=\Sy^4\Span{x_0,x_2}\;.
\]
From $x^0x^1\ins f=0$ readily follows that $f\in V_0+V_1$, that is, $f=f_0+f_1$ with $f_0\in V_0$, $f_1\in V_1$. Hence $\rk f=\rk\left(f_0+f_1\right)\le\rk f_0+\rk f_1\le 8$, because $f_0,f_1$ are polynomials in two variables. Note that, moreover, $f=\left(f_0+k{x_2}^4\right)+\left(f_1-k{x_2}^4\right)$ for all $k\in K$. Then $\rk f=8$ only if $\rk\left(f_0+k{x_2}^4\right)=\rk\left(f_1-k{x_2}^4\right)=4$ for all $k\in K$.

Let us set \[W_0:=\Span{f_0,x_1^4,x_2^4}\;,\qquad W_1:=\Span{x_0^4,f_1,x_2^4}\;.\]
If $\dim W_0=2$ or $\dim W_1=2$, then $\rk f_0\le 2$ or $\rk f_1\le 2$. Hence we can assume that $\dim W_0=\dim W_1=3$, so that \autoref{Casi} applies to both $W_0$ and $W_1$. According to the lemma,
\[
\rk\left(f_0+k{x_2}^4\right)=\rk\left(f_1-k{x_2}^4\right)=4\quad\forall k\in K
\]
can happen only in \hyperref[c2]{Case~\ref{c2}} and, more specifically, only when $f_0\in\Span{x_1{x_2}^3,{x_2}^4}, f_1\in\Span{x_0{x_2}^3,{x_2}^4}$. But in this case we have $f=f_0+f_1\in\Span{x{x_2}^3,{x_2}^4}$, with $x\in\Span{x_0,x_1}$, so that $f$ is a polynomial in two variables $x,x_2$ (actually, of rank four).
\end{proof}

\section{The general case}

\begin{prop}\label{Gener}
Let $\dim S_1=3$, $f\in S_4$. If there exist linearly independent $x^0,x^1,x^2\in S^1$ such that $x^0x^1x^2\ins f=0$ then $\rk f\le 7$.
\end{prop}
\begin{proof}
Let $\left(x_0,x_1,x_2\right)$ be the basis of $S_1$ dual to $\left(x^0,x^1,x^2\right)$ and set
\[
V_0:=\Sy^4\Span{x_1,x_2}\;,\quad V_1:=\Sy^4\Span{x_0,x_2}\;,\quad V_2:=\Sy^4\Span{x_0,x_1}
\]
($V_0,V_1,V_2\subset S_4$). Let \[\sigma:V_0\oplus V_1\oplus V_2\to V_0+V_1+V_2\] be the canonical map $\left(v_0,v_1,v_2\right)\mapsto v_0+v_1+v_2$. We have
\begin{equation}\label{Ker}
\Ker\sigma=\Span{w_0,w_1,w_2}\;,
\end{equation}
with
\[
w_0:=\left(0,{x_0}^4,-{x_0}^4\right)\;,\quad w_1:=\left({x_1}^4,0,-{x_1}^4\right)\;,\quad w_2:=\left({x_2}^4,-{x_2}^4,0\right)\;.
\]
From $x^0x^1x^2\ins s=0$ readily follows that $f\in V_0+V_1+V_2$. Then $W:=\sigma^{-1}\left(\Span{f}\right)$ is a four-dimensional vector space. For each $i\in\{0,1,2\}$, let $W_i$ be the image of $W$ in the summand $V_i$ through the projection map $V_0\oplus V_1\oplus V_2\to V_i$, and let us denote by $\alpha_i$ the restriction $W\to W_i$. For all $w\in\sigma^{-1}(f)$ we have
\begin{equation}\label{Dec}
f=f_0+f_1+f_2\;,\quad f_i:=\alpha_i(w)\in W_i\;\,\forall i\;.
\end{equation}
From \eqref{Ker} it follows that
\[
{x_1}^4,{x_2}^4\in W_0\;,\quad {x_0}^4,{x_2}^4\in W_1\;,\quad {x_1}^4,{x_2}^4\in W_2\;,
\]
hence for every decomposition~\eqref{Dec} we have
\begin{equation}\label{Rel}
W_0=\Span{f_0,{x_1}^4,{x_2}^4}\;,\quad W_1=\Span{{x_0}^4,f_1,{x_2}^4}\;,\quad W_2=\Span{{x_0}^4,{x_1}^4,f_2}
\end{equation}
(therefore $2\le\dim W_i\le 3$ for each $i$).

If $\dim W_i=2$ for some $i$, say $i=0$, then it must be $W_0=\Span{{x_1}^4,{x_2}^4}$, and we can choose a suitable $w\in\sigma^{-1}(f)$ that gives a decomposition~\eqref{Dec} with $f_0=0$. This immediately implies that $x^1x^2\ins f=0$, and the statement follows from \autoref{Wu}. Thus, from now on, we can assume that $\dim W_0=\dim W_1=\dim W_2=3$.

According to \autoref{Casi}, we get three R-configurations $\mathcal{C}_i=\left(W_i,y_i,z_i,R_i,R'_i\right)$, $i\in\{0,1,2\}$, with the obvious meaning of the notation. Note that we can use \autoref{3D}, and borrow the notation $\widehat{R_i},\widehat{R_i'}$ from there. Suppose that there exists $P\in\left(\widehat{R_0}\cap\widehat{R_1}\right)-\left(\widehat{R_0'}\cup\widehat{R_1'}\cup\widehat{R_2'}\right)$. We can certainly find a representative vector $w$ of $P$ (i.e., a generator) such that $\sigma\left(w\right)=f$. Hence we get a decomposition~\eqref{Dec} with $\rk f_0\le 2$, $\rk f_1\le 2$, $\rk f_2\le 3$, which immediately implies that $\rk f\le 7$. Thus the statement is proved whenever condition~\eqref{Vuoto} in \autoref{3D} fails for $\mathcal{C}_0,\mathcal{C}_1,\mathcal{C}_2$. According to the proposition, $\rk f\le 7$ is still to be proven only in the following two occurrences:
\renewcommand{\theenumi}{\Roman{enumi}}
\begin{enumerate}
\item\label{primo} up to possibly reordering the indices, $W_2=\Span{{x_0}^4,x_0{x_1}^3,{x_1}^4}$, $R_1$ is a reducible conic, and the plane $\widehat{R'_2}$ is a component of $\widehat{R_1'}$ (\footnote{We can exclude that $W_2=\Span{{x_0}^4,{x_0}^3x_1,{x_1}^4}$ because in this case $R'_2=\Ps\Span{{x_0}^4,{x_0}^3x_1}-\Span{{x_0}^4}$. This implies that, considering $\Ps\Ker\sigma$ as the plane at infinity of $\Ps W$, $\Span{w_1}$ can not be a point at infinity of $\widehat{R'_2}$ ($\Span{w_0}$ is). On the contrary, $\Span{w_1}$ must be a point at infinity of every component of~$\widehat{R'_1}$.}); or
\item\label{secondo} at least two among $\mathcal{C}_0,\mathcal{C}_1,\mathcal{C}_2$ are of type~\ref*{c2}.
\end{enumerate}
The workaround we shall use in these cases is basically a change of variables. Let $x'_0,x'_1,x'_2\in S_1$ be linearly independent and let $\left({x'}^0,{x'}^1,{x'}^2\right)$ be the basis of $S^1$ dual to $\left(x'_0,x'_1,x'_2\right)$. In each case, we shall choose $x'_0,x'_1,x'_2$ in such a way that the decomposition~\eqref{Dec} gives, after a linear substitution, again a decomposition of the form
\begin{equation}\label{Dec2}
f=f'_0+f'_1+f'_2\;,\qquad f'_0\in\Sy^4\Span{x'_1,x'_2}\,,\; f'_1\in\Sy^4\Span{x'_0,x'_2}\,,\; f'_2\in\Sy^4\Span{x'_0,x'_1}\;.
\end{equation}
This is equivalent to say that the choice of the new variables again gives ${x'}^0{x'}^1{x'}^2\ins f=0$. Hence we can define new spaces $W',W'_0,W'_1,W'_2$, and apply the previous analysis. In particular, for each $i$, $f'_i\in W_i$ and
\begin{equation}\label{Rel2}
W'_0=\Span{f'_0,{x'_1}^4,{x'_2}^4}\;,\quad W'_1=\Span{{x'_0}^4,f'_1,{x'_2}^4}\;,\quad W'_2=\Span{{x'_0}^4,{x'_1}^4,f'_2}\;.
\end{equation}

Let us now face \hyperref[primo]{Case~\ref{primo}}. Since $\mathcal{C}_1$ is of type~\ref*{c11} with $R_1$ containing a singular point $\Span{x}$, according to \autoref{Casi},~\ref{c11}, we have $\rk x=1$. We can choose a $w=\left(f_1,f_2,f_3\right)$ that gives a decomposition~\eqref{Dec} with $f_1\in\Span{x}$. Hence $f_1=\left(\alpha{x_0}+\beta{x_2}\right)^4$ with $\alpha,\beta\ne 0$. Since $\Span{f_1}=\Span{x}$ is contained in both components of $R_1$, we have that $w\in\widehat{R'_2}$, hence $\Span{f_2}\in R'_2=\Ps\Span{x_0{x_1}^3,{x_1}^4}-\Span{{x_1}^4}$. Up to adding to $w$ a suitable multiple of $w_0$, we can assume that $f_2=\gamma x_0{x_1}^3$, with $\gamma\ne 0$ (basically, we are moving the monomial in ${x_1}^4$ of $f_2$ into $f_0$). By rescaling $x_0,x_1,x_2$ we can further simplify:
\[
f_1=(x_0+x_2)^4\;,\quad f_2=x_0{x_1}^3\;.
\]
Now let us set $x'_0:=x_0+x_2$, $x'_1:=x_1$, $x'_2:=x_2$. By substitution we get
\[
f=f'_0+{x'_0}^4+x'_0{x'_1}^3\;,
\]
with $f'_0\in\Sy^4\Span{x'_1,x'_2}$, which can be viewed as a decomposition of the form~\eqref{Dec2} with $f'_1=0$ (\footnote{As a cross-check, note that the dual basis is ${x'}^0=x^0,{x'}^1=x^1,{x'}^2=-x^0+x^2$, and indeed ${x'}^0{x'}^1{x'}^2\ins f=0$ (in the present case ${x^0}^2x^1\ins f=0$ because $f_2=x_0{x_1}^3$).}). Hence $\dim W'_1=2$, and we already know that $\rk f\le 7$ in such a case.

We are left with \hyperref[secondo]{Case~\ref{secondo}}. We can assume that (up to possibly reordering the indices) the R-configurations $\mathcal{C}_0$ and $\mathcal{C}_1$ are of type~\ref*{c2}. We have to consider the following subcases:
\renewcommand{\theenumi}{\roman{enumi}}
\begin{enumerate}
\item\label{i} $W_0=\Span{{x_1}^4,{x_1}^3x_2,{x_2}^4}$, $W_1=\Span{{x_0}^4,{x_0}^3x_2,{x_2}^4}$;
\item\label{ii} $W_0=\Span{{x_1}^4,{x_1}^3x_2,{x_2}^4}$, $W_1=\Span{{x_0}^4,x_0{x_2}^3,{x_2}^4}$, or $W_0=\Span{{x_1}^4,x_1{x_2}^3,{x_2}^4}$, $W_1=\Span{{x_0}^4,{x_0}^3x_2,{x_2}^4}$;
\item\label{iii} $W_0=\Span{{x_1}^4,x_1{x_2}^3,{x_2}^4}$, $W_1=\Span{{x_0}^4,x_0{x_2}^3,{x_2}^4}$.
\end{enumerate}

We preliminary also assume that $\mathcal{C}_2$ is not of type~\ref*{c2} (the opposite case will be discussed at the end). 

In \hyperref[i]{Case~\ref{secondo},~\ref{i}}, if ${x^2\,}^4\ins f\ne 0$ (that is, the monomial ${x_2}^4$ occurs with a nonzero coefficient in $f$, considered as a polynomial in $x_0,x_1,x_2$), let us set $x'_0:=x_0$, $x'_1:=x_1$, $x'_2=kx_0+x_2$. Taking into account~\eqref{Rel}, we readily get from \eqref{Dec} a decomposition in the form~\eqref{Dec2}. Taking into account \eqref{Rel2}, we also can fix $k$ in such a way that neither $W'_1$ nor $W'_2$ gives an R-configuration of type~\ref*{c2} (recall we are also assuming that $\mathcal{C}_2$ is not of type~\ref*{c2}). Hence in the new variables we fall outside \hyperref[secondo]{Case~\ref{secondo}}, so that $\rk f\le 7$ has already been proved.

Still considering \hyperref[i]{Case~\ref{secondo},~\ref{i}}, but now with ${x^2\,}^4\ins f=0$, the appropriate substitution is of the form $x'_0:=x_0$, $x'_1:=x_1$, $x'_2=hx_0+kx_1+x_2$. The key point here is that we can choose $h,k$ and further scalars $\alpha,\beta$ in such a way that $g'_2:=f'_2+\alpha {x'_0}^4+\beta {x'_1}^4$ is a $4$-th power of a linear form (details are not difficult and left to the reader), hence its rank is at most one. Moreover, for a generic choice of $\gamma\in K$, the rank of both polynomials $g'_0:=f'_0-\alpha {x'_1}^4+\gamma {x'_2}^4$ and $g'_1:=f'_1-\beta {x'_0}^4-\gamma {x'_2}^4$ is at most three. Hence the rank of $f=g'_0+g'_1+g'_2$ is at most $3+3+1=7$, as required.

In \hyperref[ii]{Case~\ref{secondo},~\ref{ii}}, up to possibly exchanging the indices $0,1$, we can assume that $W_0=\Span{{x_1}^4,{x_1}^3x_2,{x_2}^4}$ and $W_1=\Span{{x_0}^4,x_0{x_2}^3,{x_2}^4}$. Here we can proceed exactly as in the subcase~\ref{i} when ${x^2\,}^4\ins f\ne 0$ (without any need of this restrictive assumption, because of the presence of $x_0{x_2}^3$ in $f_1$).

In \hyperref[iii]{Case~\ref{secondo},~\ref{iii}} it suffices to set $x'_0=x_0+kx_1$, $x'_1=x_1$, $x'_2=x_2$ and choose $k$ in such a way that $\dim W'_0=2$ (which gives an already settled case).

Note that \hyperref[secondo]{Case~\ref{secondo}} is now solved whenever we have \emph{exactly two} R-configurations of type~\ref*{c2}. The only event left is when all R-configurations are of type~\ref*{c2}. With reference to the previous discussion of the subcases~\ref{i}, \ref{ii}, \ref{iii}, we needed that $\mathcal{C}_2$ is not of type~\ref*{c2} only in \hyperref[i]{Case~\ref{secondo},~\ref{i}} with ${x^2\,}^4\ins f\ne 0$ and in \hyperref[ii]{Case~\ref{secondo},~\ref{ii}}, in order to assure that $k$ could be chosen in such a way neither $W'_1$ nor $W'_2$ gave an R-configuration of type~\ref*{c2}. But in both situations, $k$ can be chosen in such a way that $W'_1$ do not give an $R$-configuration of type~\ref*{c2}. If $W'_2$ does, we nevertheless fall into the exactly two type~\ref*{c2} R-configurations case, which is now solved.
\end{proof}

\section{On reduction to the general case}

At the end of the Introduction, we explained how to find triples $x^0,x^1,x^2\in S^1$ such that $x^0x^1x^2\ins f=0$ ($\dim S_1=3$, $f\in S_4$). In the notation there, the set of all such $\Span{x^0x^1x^2}\in\Ps S^3$ is an algebraic set $X\cap Y$ of dimension at least three. Since $Y$ is a very special subspace, one can hope it will not also give a very special intersection with $X$. That is, an intersection that entirely falls within the special locus corresponding to linearly dependent $x^0,x^1,x^2$. The following simple result encourages this expectation.

\begin{prop}\label{Reduc}
Let $\dim S_1=3$, $f\in S_4$. There exist distinct $\Span{x^0},\Span{x^1},\Span{l}\in\Ps S^1$ such that $x^0x^1l\ins f=0$.
\end{prop}
\begin{proof}
The case $f=0$ being trivial, let us assume $f\ne 0$. Since the image of the (vector) Veronese map $S^1\to S^4$, $x\mapsto x^4$, spans $S^4$, we can fix $x^0\in S^1$ with \[{x^0\,}^4\ins f\ne 0\;.\] It follows, in particular, that $g:=x^0\ins f\ne 0$. The dimension of $V:=\Ker g_{2,1}$ is at least three because $g_{2,1}$ maps $S^2$ into $S^1$. Since the locus $X\subset\Ps S^2$ given by reducible forms is a hypersurface, we have that the intersection $Y:=\Ps V\cap X$ is an algebraic set of dimension at least one. For distinct $\Span{x},\Span{y}\in\Ps S^1$, we have that $\Span{x^2},\Span{y^2}\in Y$ implies $\Span{x^2+y^2}\in Y$, and $x^2+y^2$ is a simply degenerate quadratic form ($\operatorname{char}K=0\ne 2$). We deduce that the set $U$ of $\Span{x^1l}\in Y$ with distinct $\Span{x^1},\Span{l}\in\Ps S^1$, is a nonempty open subset of $Y$. But $\Span{x^1l}\in U\subseteq\Ps L$ means that $x^1l\ins g=0$, and $x^1l\ins g=x^0x^1l\ins f$. Hence it remains only to prove that we can choose $\Span{x^1},\Span{l}$ different from $\Span{x_0}$.

Suppose then that all $q\in U$, and hence all $q\in Y$ are divisible by $x_0$. We can choose two distinct $\Span{y},\Span{z}\in\Ps S^1$ such that $\Span{x^0y},\Span{x^0z}\in Y$. Let \[q\in V-\Span{x^0y,x^0z}\;.\]
It can not be $q=x^0w$ with $w\in S^1$, otherwise $y$, $z$ and $w$ would be linearly independent, and hence $x^0\in\Span{y,z,w}$ (this would lead to ${x^0\,}^2\in V$, and henceforth ${x^0\,}^3\ins f=0$, meanwhile ${x^0\,}^4\ins f\ne 0$). If $q\in Y$ we can find $q'\in U-\Span{x^0y,x^0z}$, so that $q'=x^1l$ with $\Span{x^0},\Span{x^1},\Span{l}$ all distinct as required.

We are left with the case when $q\not\in Y$, so that $q$ is nondegenerate. Note that the (distinct) lines $y=0$ and $z=0$ in $\Ps{S_1}$ do not meet at a point lying on the line $x^0=0$, otherwise $x^0\in\Span{y,z}$ which would lead to ${x^0\,}^3\ins f=0$ as before. Hence we can find $w\in\Span{y,z}$ such that the line $w=0$ meets $q=0$ in distinct points that are also outside the line $x^0=0$. Now $\Ps\Span{q,x_0w}\subset\Ps V$ gives a pencil of conics in $\Ps S_1$. Looking at its base points, we easily deduce the existence of a simply degenerate form $x^1l$, with $\Span{x^1},\Span{l}$ both different from $\Span{x_0}$ as required.
\end{proof}

We tried to refine the above arguments to get \emph{linearly independent} $x^0,x^1,l$ with $x^0x^1l\ins f=0$. But, in view of our goal, we found easier to adapt the proof of~\autoref{Gener} to the special case when $l\in\Span{x^0,x^1}$ and $\Span{x^0},\Span{x^1},\Span{l}\in\Ps S^1$ all distinct. At the moment of writing, we do not know if linearly independent $x^0,x^1,x^2\in S^1$ with $x^0x^1x^2\ins f=0$ can be found for \emph{every} $f\in S_4$.

\section{The special case}

The following proposition is about the special case of linearly dependent $x^0,x^1,l$ (but with $\Span{x^0},\Span{x^1},\Span{l}\in\Ps S^1$ all distinct, i.e., $x^0,x^1,l$ are pairwise linearly independent). This can be proved much like \autoref{Gener}, but at the cost of leaving out a (more) special case, which still needs work. That is why below we are adding the hypothesis that ${x^1\,}^2\ins f$, ${x^2\,}^2\ins f$ and $l^2\ins f$ do not vanish. We shall show how to remove this hypothesis at the end of this section.

\begin{prop}\label{Spe}
Let $\dim S_1=3$, $f\in S_4$. If there exist linearly dependent, but pairwise linearly independent, $x^0,x^1,l\in S^1$, such that $x^0x^1l\ins f=0$, but ${x^1\,}^2\ins f$, ${x^2\,}^2\ins f$ and $l^2\ins f$ are all nonzero, then $\rk f\le 7$.
\end{prop}
\begin{proof}
Let us choose $y\in\{x^0,x^1\}^\perp-\{0\}$, $x_1\in\{x^0\,\}^\perp-\Span{y}$ with $l\ins x_1=1$, $x_0\in\{x^1\,\}^\perp-\Span{y}$ with $l\ins x_0=1$ and set
\begin{equation}\label{V}
V_0:=\Sy^4\Span{x_1,y}\;,\quad V_1:=\Sy^4\Span{x_0,y}\;,\quad V_2:=\Sy^4\Span{x_0-x_1,y}
\end{equation}
($V_0,V_1,V_2\subset S_4$) (\footnote{The notation $V_2$ is slightly misleading, but it speeds up the exposition.}). Let \[\sigma:V_0\oplus V_1\oplus V_2\to V_0+V_1+V_2\] be the canonical map $\left(v_0,v_1,v_2\right)\mapsto v_0+v_1+v_2$. We have
\begin{equation}\label{Ker2}
\Ker\sigma=\Span{w_0,w_1,v}\;,
\end{equation}
with
\[
w_0:=\left(0,{x_2}^4,-y^4\right)\;,\quad w_1:=\left(y^4,0,-y^4\right)\;,\quad v:=\left(x_1y^3,-x_0y^3,\left(x_0-x_1\right)y^3\right)\;.
\]
From $x^0x^1l\ins s=0$ follows that $f\in V_0+V_1+V_2$. Then $W:=\sigma^{-1}\left(\Span{f}\right)$ is a four-dimensional vector space. For each $i\in\{0,1,2\}$, let $W_i$ be the image of $W$ in the summand $V_i$ through the projection map $V_0\oplus V_1\oplus V_2\to V_i$, and let us denote by $\alpha_i$ the restriction $W\to W_i$. For all $w\in\sigma^{-1}(f)$ we have
\begin{equation}\label{Dec3}
f=f_0+f_1+f_2\;,\quad f_i:=\alpha_i(w)\in W_i\;\,\forall i\;.
\end{equation}
From \eqref{Ker2} it follows that
\[
x_1y^3,y^4\in W_0\;,\quad x_0y^3,y^4\in W_1\;,\quad \left(x_0-x_1\right)y^3,y^4\in W_2\;.
\]
Hence for every decomposition~\eqref{Dec3} we have
\[
W_0=\Span{f_0,x_1y^3,y^4}\;,\quad W_1=\Span{x_0y^3,f_1,y^4}\;,\quad W_2=\Span{y^4,\left(x_0-x_1\right)y^3,f_2}
\]
(therefore $2\le\dim W_i\le 3$ for each $i$).

If $\dim W_i=2$ for some $i$, we can choose $w$ such that the decomposition~\eqref{Dec3} becomes
\[
f=g\left(z,y\right)+h\left(t,y\right)
\]
with $y,z,t\in S_1$ linearly independent, and the result follows from \autoref{Wu}.

From now on, we can assume that $\dim W_i=3$ for all $i$. Then we can exploit \autoref{Casi2} for each $i$ and get varieties $R_i$, $R'_i$ (with the obvious meaning of the notation). For each $i$, let
\[
\widehat{\alpha_i}:\Ps W-\Ps{\Span{w_0,w_1,v}}\longrightarrow\Ps{W_i}\;,\qquad \Span{w}\mapsto\Span{\alpha_i(w)}
\]
and set
\[
\widehat{R_i}:=\widehat{\alpha_i}^{-1}\left(R_i\right)\;,\qquad\widehat{R_i'}:=\widehat{\alpha_i}^{-1}\left(R_i'\right)\;.
\]
We are now in a situation similar to that of \autoref{3D}, and the loci $\widehat{R_i},\widehat{R'_i}$ are cylinders with vertices the (aligned, at infinity) points $\Span{w_0},\Span{w_1},\Span{w_0-w_1}$. As in that situation, the analysis can be pursued in different ways, one of which we outline as follows.

The good news brought by \autoref{Casi2} is that $R'_i$ always consists of at most two points. Suppose first that for $W_0$ we fall in \hyperref[d11]{Case~\ref{d11}} of \autoref{Casi2} with $R_0$ degenerate, so that there exists
\[
\Span{z}\in\Ps W_0-\Ps\Span{x_1y^3,y^4}
\]
with $\rk z=1$. Then we can choose $w\in\Span{{\alpha_0}^{-1}(z),w_0}-\Span{w_0}$ such that $\rk\alpha_1(w)\le 3$, $\rk\alpha_2(w)\le 3$. Hence \eqref{Dec3} for such a $w$ gives $\rk f\le 1+3+3=7$. The same argument works if for $W_1$ (or even for $W_2$) we fall into \hyperref[d11]{Case~\ref{d11}} of \autoref{Casi2} with $R_1$ (or, respectively, $R_2$) degenerate. With these cases excluded, it is not difficult to check that if at most one among $W_0,W_1,W_2$, say $W_i$, leads to \hyperref[d2]{Case~\ref{d2}} in \autoref{Casi2}, then we have $\left(\widehat{R_j}\cap\widehat{R_k}\right)-\left(\widehat{R_0'}\cup\widehat{R_1'}\cup\widehat{R_2'}\right)\ne\emptyset$, with $j,k$ being the two indices other than~$i$. This clearly gives $\rk f\le 7$ (as in the proof of \autoref{Gener}).

Now, we can assume that for $W_0,W_1$ we fall in \hyperref[d2]{Case~\ref*{d2}} of \autoref{Casi2}, that is,
\[
W_0=\Span{{x_1}^2y^2,x_1y^3,y^4}\;,\qquad W_1=\Span{{x_0}^2y^2,x_0y^3,y^4}\;.
\]
Let us fix a decomposition~\eqref{Dec3} (corresponding to some $w$). Note that, by the choices of $x_0,x_1,y$ at the beginning of the proof and by~\eqref{V}, we have $x^0\ins f_0=x^1\ins f_1=l\ins f_2=x^0\ins y=x^1\ins y=l\ins y=0$. From~\eqref{Li} easily follows that
\[
x^0l\ins f=\alpha y^2\;,\qquad x^1l\ins f=\beta y^2\;,
\]
for some $\alpha,\beta\in K$. Hence $(\beta x^0-\alpha x^1)l\ins f=0$. Since $l^2\ins f\ne 0$, $(\beta x^0-\alpha x^1)$ and $l$ are linearly independent, so that $\rk f\le 7$ follows from \autoref{Wu}.
\end{proof}

Basically, the analysis in the above proof stopped when facing a very special $f$ (such that two among $W_0,W_1,W_2$ fall in \hyperref[d2]{Case~\ref*{d2}}, and after reordering $x^0,x^1,l$ accordingly, we have that $\beta x^0-\alpha x^1$ is proportional to~$l$). In order to settle this and then reach our goal of giving a new proof that $\rk f\le 7$, we now (more generally) work out the condition $l^2\ins f=0$. This way, the result will again be included, as \autoref{Wu}, in \cite[Theorem~3.6]{K}, but in this case we propose a proof which looks different (and fits into the approach of the present work).

\begin{prop}\label{Fine}
Let $\dim S_1=3$, $f\in S_4$. If there exist a nonzero $l\in S^1$ such that $l^2\ins f=0$ then $\rk f\le 7$.
\end{prop}
\begin{proof}
The dimension of $V:=\Ker f_{3,1}$ is at least $7$, because $f_{3,1}$ maps $S^3$ into $S_1$. Let $W:=V\cap lS^2$, so that $\dim W\le 6$. If $\dim W=6$ then $l\ins f=0$ and therefore $f\in\Sy^4\{l\}^\perp$, so that $\rk f\le 4$. If $\dim W=5$ then $g:=l\ins f$ is of rank one because its polarization $g_{2,1}$ must be of rank one. This means that $g=z^3$ for some $z\in S_1$, and $l\ins z^3=l^2\ins f=0$. If we take $y\in S_1$ such that $l\ins y=1$, we have $f=yz^3+h$ with $h\in\Sy^4\{l\}^\perp$. Therefore, for whatever nonzero $m\in\{y,z\}^\perp$ we have $lm\ins f=0$ and $l,m$ are linearly independent because $l\ins y=1$, $m\ins y=0$. Hence the result follows from \autoref{Wu}.

From the above, now we can assume that $\dim W\le 4$. Also recall that $\dim V\ge 7$. Therefore the image of $V$ under the projection map $\pi:S^\bullet\to S^\bullet/(l)$ (with $(l)$ being the ideal generated by $l$) is of dimension at least three. It easily follows that there exists $p\in V$ such that the cubic $p=0$ in $\Ps S_1$ intersect the line $l=0$ in three distinct points $P_0,P_1,P_2$. To be concise, we now use a bit of elementary scheme-theoretical language. The scheme-theoretic intersection $Z$ of $p=0$ with the double line $l^2=0$ consists of $P_0,P_1,P_2$ doubled \emph{inside} three lines $x^0=0$, $x^1=0$, $x^2=0$ (`a point $P$ doubled inside a line $\ell$' is the degree two zero-dimensional scheme with ideal sheaf $\mathcal{I}_P^2+\mathcal{I}_\ell$). It is easy to see that the ideal of $Z$ in $S^\bullet$ is generated by $p,l^2$, so that $x^0x^1x^2=\alpha p+xl^2$, with $\alpha\in K$, $x\in S^1$. It follows that $x^0x^1x^2\ins f=0$, and $\Span{x^0}$, $\Span{x^1}$, $\Span{x^2}$ are distinct because the lines $x^i=0$ meet $l=0$ in distinct points. Therefore, in view of \autoref{Gener} and \autoref{Spe}, we can assume that for some $i$ we have ${x^i\,}^2\ins f=0$. But $\Span{x^i}\ne\Span{l}$, because the lines $x^i=0$ and $l=0$ meet only at $P_i$. Hence $l+x^i$ and $l-x^i$ are linearly independent,
\[
\left(l-x^i\right)\left(l+x^i\right)\ins f=\left(l^2-{x^i\,}^2\right)\ins f=0\;,
\]
and the result follows from \autoref{Wu}.
\end{proof}

Propositions~\ref{Reduc}, \ref{Gener}, \ref{Spe}, and~\ref{Fine} together give a bound of seven for \emph{every} plane quartic. Since it is well known that a nondegenerate conic together with a doubled tangent line gives a rank seven plane quartic, we end up with Kleppe's result that the maximal rank for plane quartics is seven (which solves the polynomial little Waring problem for $n=3$, $d=4$).


\begin{thebibliography}{5}

\bibitem{BB}
{Bia\l ynicki-Birula}, A.\ and {Schinzel}, A.
\newblock {Representations of multivariate polynomials by sums of univariate
  polynomials in linear forms.}
\newblock {\em {Colloq. Math.}}, 112(2):201--233, 2008.

\bibitem{BB2}
{Bia\l ynicki-Birula}, A.\ and {Schinzel}, A.
\newblock Corrigendum to ``{R}epresentatons of multivariate polynomials by sums
  of univariate polynomials in linear forms'' ({C}olloq. {M}ath. 112 (2008),
  201--233).
\newblock {\em Colloq. Math.}, 125(1):139, 2011.

\bibitem{BGI}
Bernardi, Alessandra, Gimigliano, Alessandro and Id\`a, Monica.
\newblock {Computing symmetric rank for symmetric tensors.}
\newblock {\em J. Symb. Comput.}, \textbf{46(1)} (2011) 34--53.

\bibitem{CS}
Comas, Gonzalo and Seiguer, Malena.
\newblock {On the rank of a binary form.}
\newblock {\em Found. Comput. Math.}, \textbf{11(1)} (2011) 65-78.

\bibitem{G}
Geramita, Anthony~V.
\newblock {Expos\'e I A: Inverse systems of fat points: Waring's problem,
  secant varieties of Veronese varieties and parameter spaces for Gorenstein
  ideals.}
\newblock In: {\em The {C}urves {S}eminar at {Q}ueen's, {V}ol.\ {X} ({K}ingston,
  {ON}, 1995)}. {Queen's University, Kingston}, (1996) 2--114.
  
\bibitem{J}
Jelisiejew, Joachim.
\newblock An upper bound for the waring rank of a form.
\newblock Preprint arXiv:1305.6957v1 [math.AG], available at \/ {\tt
  http://arxiv.org/abs/1305.6957v1}.

\bibitem{K}
Kleppe, Johannes.
\newblock {Representing a Homogenous Polynomial as a Sum of Powers of Linear Forms.}
\newblock Thesis for the degree of Candidatum Scientiarum, Department of Mathematics, Univ. Oslo (1999).

\bibitem{L}
Landsberg, Joseph M.
\newblock {\em {Tensors: Geometry and applications.}}
\newblock {American Mathematical Society (AMS), Providence, RI}, 2012.

\end{thebibliography}
\end{document}